\documentclass{amsart}

\usepackage{graphicx}
\usepackage{enumerate}
\usepackage{amsmath,amssymb}
\usepackage{mathrsfs}
\usepackage{color}
\usepackage{bm}
\usepackage{marginnote}

\newtheorem{theorem}{Theorem}[section]
\newtheorem{lemma}[theorem]{Lemma}
\newtheorem{proposition}[theorem]{Proposition}
\newtheorem{corollary}[theorem]{Corollary}
\newtheorem{remark}[theorem]{Remark}

\newtheorem*{question}{Question}

\newtheorem*{add}{Addendum}

\numberwithin{equation}{section}
\numberwithin{figure}{section}

\def\loc{\mathrm{loc}}

\def\per{\mathrm{Per}}

\def\uind{\mathrm{u}\text{-}\mathrm{ind}}

\begin{document}

%%%%%%%%%%%%%
\title[Wandering domains for heterodimensional cycles
]
{
Non-trivial wandering domains 
for heterodimensional cycles
}

\author{Shin Kiriki} 
\address{
Department of Mathematics, Tokai University, 4-1-1 Kitakaname, Hiratuka, Kanagawa, 259-1292, 
JAPAN}
\email{kiriki@tokai-u.jp}

\author{Yushi Nakano}
\address{Kitami Institute of Technology, 165 Koen-cho, Kitami, Hokkaido 090-8507, JAPAN}
\email{nakano@mail.kitami-it.ac.jp}

\author{Teruhiko Soma}
\address{Department of Mathematics and Information Sciences,
Tokyo Metropolitan University,
Minami-Ohsawa 1-1, Hachioji, Tokyo 192-0397, JAPAN}
\email{tsoma@tmu.ac.jp}


\date{\today}
%\date{version 0814}

\dedicatory{
Dedicated to Professor Iku Nemoto on his 70th birthday.
}

%%%%%%%%%%%%%

\maketitle

\begin{abstract}
We present a sufficient condition 
for three-dimensional diffeomorphisms having heterodimensional cycles
to be approximated 
arbitrarily well by diffeomorphisms with non-trivial contracting 
wandering domains via several perturbations. The key idea  is to show 
that diffeomorphisms with heterodimensional cycles associated with 
saddle points with non-real eigenvalues can be approximated by 
diffeomorphisms with generalized homoclinic tangencies presented by 
Tatjer. The generalized homoclinic tangency is an organizing center including a 
Bogdanov-Takens bifurcation, by which one can obtain non-trivial 
contracting wandering domains together with a Denjoy-like construction. 
\end{abstract}

\section{Introduction}\label{s.Introduction}
 \subsection{Main result}
In this paper we study the 
existence of non-trivial wandering domains in nonhyperbolic dynamics.
Here  a \emph{non-trivial wandering domain} for a given map $f$ on a  manifold $M$ means a 
non-empty connected open set $D\subset M$ which satisfies the following conditions:
\begin{itemize}
\item $f^{i}(D)\cap f^{j}(D)= \emptyset$ for every $i, j\geq 0$ with $i\neq j$;
\item the union of the  $\omega$-limit sets of  points in $D$ for $f$, denoted by $\omega(D, f)$, is not equal to a single periodic orbit.
\end{itemize}
See \cite[p.\ 36]{dMvS93} for the original definition in the one-dimensional case. 
A wandering domain $D$ is called \emph{contracting} if the diameter of $f^{n}(D)$  
converges to zero as $n\to \infty$. 
In the early 20th century, 
Bohl \cite{Bo16} and Denjoy \cite{D32} constructed 
examples of 
$C^1$ diffeomorphisms on a circle which have contracting wandering domains in which 
the union of the $\omega$-limit sets of points is a Cantor set.
Following these results, 
similar phenomena  
were observed for high dimensional examples, see \cite{Kn81, Ha89, Mc93, BGLT94, NS96, KM10}. 
On the other hand,  the absence of wandering domains
is the key of classification of one-dimensional  unimodal as well as one-dimensional multimodal maps,  
in real analytic category, which were 
developed in \cite{dMvS89, Ly89,BlLy89, dMvS93, vSV04}, 
see the survey of van\ Strien \cite{vS10}.   
For wandering domains of rational maps of the Riemann sphere, see \cite{Su, MdMvS92}.
Moreover, Berry and Mestel \cite{BM91} showed that 
any $C^{1}$ Lorenz map without gaps does not 
admit a wandering domain, 
but the corresponding assertion for the contracting case with gaps has not been shown yet, see \cite{GC12}.

Topics about the existence of non-trivial wandering domains in nonhyperbolic  dynamics
were first studied by Colli-Vargas \cite{CV01} for some two-dimensional  example
which is made up of an affine thick horseshoe with $C^{2}$-persistent homoclinic tangencies. 
Moreover, their conjecture was recently proved to be true 
by the first and third authors \cite[Theorem A]{KS-ax}: 
any two-dimensional diffeomorphism in any Newhouse open set 
is contained in the $C^{r}$ ($2\leq r<\infty$)  closure 
of diffeomorphisms having contracting non-trivial wandering domains 
for which the union of the $\omega$-limit sets of points contains a basic set.
This result moreover implies an affirmative answer in the $C^{r}$  ($2\leq r<\infty$)  category to one of the open problems 
of van Strien \cite{vS10} which is concerned with the existence of wandering domains for the H\'enon family, see  \cite[Corollary B]{KS-ax}.
Compare with a negative answer in the $C^{\omega}$ category given in \cite{Ou}.

There is another well-studied nonhyperbolic phenomenon different from a homoclinic tangency,  
which is called a heterodimensional cycle. 
We say that  a diffeomorphism  has  a  \emph{heterodimensional cycle} associated with 
saddle periodic points 
if there are saddle periodic points 
$P$ and $Q$ for the diffeomorphism such that
$$  
W^{u}(P)\cap W^{s}(Q)\neq \emptyset,\ 
W^{u}(Q)\cap W^{s}(P)\neq \emptyset,\ 
\uind(P)\neq  \uind(Q),$$
where $\uind(\cdot)$ is the dimension of the unstable bundle, 
called the \emph{unstable index}, for a corresponding periodic point. 
Thus a natural question is the following:
 \begin{question}
Let $f$ be a diffeomorphism having a heterodimensional cycle.
Is the diffeomorphism $f$ contained in the $C^{r}$ closure 
of diffeomorphisms having contracting non-trivial wandering domains?
 \end{question}
The next theorem is one of the main results in the present paper, which gives 
an affirmative answer to the question in the $C^1$-category. 
 \begin{theorem}\label{thm1}
Let $f$ be a diffeomorphism on a three-dimensional manifold which 
has a heterodimensional cycle 
associated with two saddle periodic points at which both the derivatives for $f$ have non-real eigenvalues.
Then there exists a diffeomorphism $g$ arbitrarily $C^{1}$-close to $f$ 
such that $g$ has a contracting non-trivial wandering domain $D$ and 
$\omega(D, g)$ is a nonhyperbolic transitive Cantor set without periodic points. 
 \end{theorem}

Note that, in  \cite{CV01, KS-ax}, to construct non-trivial wandering domains near a 
homoclinic tangency of a 2-dimensional diffeomorphism $F$, they added a 
countable 
series of perturbations to $F$ supported on some open sets which are 
respectively 
contained in mutually disjoint gaps in the complement of persistent 
tangencies. 
On the other hand in this paper, to show Theorem  \ref {thm1} we adopt a quite different procedure as follows:
 \begin{description}
\item[step 1]  $C^1$-approximate
the diffeomorphism $f$  given in Theorem \ref{thm1}
by $C^r$, $r\geq 2$, diffeomorphisms with so called ``non-transverse equidimensional cycles";
\item[step 2] 
$C^{r}$-approximate 
the diffeomorphisms with the non-transverse equidimensional cycles 
by other diffeomorphisms having so called ``generalized homoclinic tangencies";
\item[step 3]  
Owing to one of Tatjer's results  \cite{Tj01}, the generalized homoclinic tangencies lead to Bogdanov-Takens bifurcations, which create invariant circles;
\item[step 4]  
Finally, by using a $C^1$ Denjoy-like construction along the invariant circle, 
one can obtain a diffeomorphism $g$ satisfying the conditions as in Theorem  \ref {thm1}.
\end{description} 
 In the next subsection, we give essential ingredients to carry out these steps.
 
 \subsection{Outline of the paper}
Denote by $\mathscr{A}$ the set of all diffeomorphisms which satisfy the assumption of Theorem \ref{thm1}.
Furthermore denote by $\mathscr{Z}$  the set of all 
 diffeomorphisms which satisfy the conclusion of Theorem \ref{thm1}.
Then, Theorem \ref{thm1} can be rephrased as follows:
\begin{corollary}\label{co1}
$\mathscr{A}$ is contained in the $C^{1}$ closure of $\mathscr{Z}$.
\end{corollary}

To show Theorem \ref{thm1}, we have to prepare auxiliary 
classes $\mathscr{B}$, $\mathscr{C}$, $\mathscr{D}$
of  diffeomorphisms on $M$ satisfying the following inclusion relations:  
$$
\mathscr{A}\subset 
\overline{(\mathscr{B})}_{C^{1}}, \quad 
\mathscr{B}\subset \overline{(\mathscr{C})}_{C^{r}}\subset 
\overline{(\mathscr{D})}_{C^{r}}, \quad 
\mathscr{D}\subset \overline{(\mathscr{Z})}_{C^{1}}, 
$$
where $\overline{( \cdot )}_{C^{1}}$ and $\overline{( \cdot )}_{C^{r}}$
respectively stand for the $C^{1}$ and $C^{r}$, $r\geq 2$,  
closures of the corresponding sets. 
We will give the definition of each class in the following sections. 
So, we here briefly explain  
what role each class plays in the proof of Theorem \ref{thm1}.

Section \ref{sec2} contains  step 1 
where we show that 
any element of $\mathscr{A}$ leads to a heterodimensional cycle containing an 
``intrinsic tangency''. It is 
a non-transverse intersection between some invariant manifold and some leaf of an invariant foliation
which is contained in some transverse heterodimensional intersections, see
Lemma \ref{lem:p5c}. Moreover, 
the intrinsic tangency yields the class $\mathscr{B}$ of $C^r$ diffeomorphsims 
for which each of elements has a
``non-transverse equidimensional cycle'',  see Proposition \ref{prop2.1}.

Section \ref{sec3} corresponds to steps 2 and 3. 
We here show  that $\mathscr{B}$ is contained in the $C^{r}$ closure of the class $\mathscr{C}$ of $C^r$ diffeomorphsims
where every element has 
a ``generalized homoclinic tangency'' presented by Tatjer \cite{Tj01}. See Proposition \ref{thm3.1}.
In short, the generalized homoclinic tangency is a certain type of
non-transverse codimension-two intersection which satisfies
the Tatjer condition. 
Here the Tatjer condition consists of the geometric properties given in (\ref{C1})--(\ref{C3}) 
of section \ref{sec3}.
As a matter of fact, 
a couple of propositions \ref{prop2.1} and \ref{thm3.1} 
is the key of this paper which directly implies another main result: 
\begin{theorem}
Every element of $\mathscr{A}$ can be arbitrarily $C^{1}$-approximated by $C^r$ diffeomorphisms 
having 
a generalized homoclinic tangency.
\end{theorem}

Note that Tatjer 
presents several types of generalized homoclinic tangencies,
which provide various phenomena according to the types, e.g.,  several types of limit return (H\'enon-like) maps of renormalizations, 
birth of attracting or saddle type invariant circles via Bogdanov-Takens bifurcations, 
the existence of strange attractors and the Newhouse phenomenon. See \cite[Theorem 1]{Tj01}.
Indeed, by virtue of one of them, 
we can find the class $\mathscr{D}$ of diffeomorphisms which have attracting invariant circles 
created by the Bogdanov-Takens bifurcation. 
 
Section  \ref{sec4} contains step 4 where we finally perform a Denjoy-like construction for a normal tubular 
neighborhood of the attracting invariant circle of any diffeomorphism in $\mathscr{D}$ 
to detect non-trivial wandering domains.  
\smallskip

In closing, we note that all approximations in this paper can be 
done in $C^r$-category with any integer $r \geq 2$, except for Lemma 2.2 
in step 1 and Proposition 4.2 in step 4.

\section{Intrinsic tangencies of cycles}\label{sec2}
Let $M$ be a three-dimensional Riemannian manifold and 
$N_{1}$ and $N_{2}$ be submanifolds of $M$.
We say that 
a point $\boldsymbol{x}\in N_{1}\cap N_{2}$ is a \emph{transverse intersection}
if it satisfies $T_{\boldsymbol{x}}M=T_{\boldsymbol{x}}N_{1} + T_{\boldsymbol{x}}N_{2}$.
Denote by $N_{1} \pitchfork N_{2}$ the set of all transverse intersections of $N_{1}$ and $N_{2}$.
On the other hand, 
a point $\boldsymbol{y}\in N_{1}\cap N_{2}$
satisfying $T_{\boldsymbol{y}}M\neq T_{\boldsymbol{y}}N_{1} + T_{\boldsymbol{y}}N_{2}$ 
is called a  \emph{tangency}.
In this section we consider the set $\mathscr{B}$ 
of  all $C^{r}$, $r\geq2$,  diffeomorphisms on $M$ for which any element 
has a \emph{non-transverse equidimensional cycle}, that is, for any $f\in\mathscr{B}$, 
there are saddle periodic points $P$ and $P^{\prime}$ for $f$  
satisfying the following conditions:
\begin{enumerate}
\renewcommand{\theenumi}{B\arabic{enumi}}
\renewcommand{\labelenumi}{(\theenumi)}
\item \label{B1}
$\uind(P)=\uind(P^{\prime})= 2$, and   the
unstable eigenvalues of $P$ are non-real, while the unstable eigenvalues of $P^{\prime}$ are real;
\item \label{B2}
 $P$ and $P^{\prime}$
are \emph{homoclinically related} to each other, i.e., $W^{s}(P^{\prime})\pitchfork W^{u}(P)\neq \emptyset$ and 
$W^{u}(P^{\prime})\pitchfork W^{s}(P)\neq \emptyset$;
\item \label{B3}
there is a  quadratic tangency between 
$W^{s}(P^{\prime})$ and $ W^{u}(P)$.
\end{enumerate}
Here the tangency $\boldsymbol{y}\in W^{s}(P^{\prime})\cap W^{u}(P)$ is said to be \emph{quadratic} (or a \emph{contact of order} $1$)  if there exist
an arc  $\ell\subset W^{s}(P^{\prime})$, a regular surface $S\subset W^{u}(P)$, and some $C^2$ change of coordinates on an open neighborhood $U(\boldsymbol{y})$ of $\boldsymbol{y}$ 
such that 
(i) $\boldsymbol{y}= (0,0,0)$, $S=\{(x,y,z)\in U(\boldsymbol{y})\,;\, z = 0\}$; 
(ii) $\ell $ has a regular parametrization $\ell(t)=(x(t),y(t),z(t))$ with $\ell(0)=(0,0,0)$; 
(iii) $z^\prime(0)=0$ and $z^{\prime\prime}(0)\neq0$.

The main result of this section is the following proposition.
\begin{proposition}\label{prop2.1}
$\mathscr{A}$ is contained in the $C^{1}$ closure of $\mathscr{B}$.
\end{proposition}

Since our setting is in dimension three,
the following lemma can be 
immediately obtained from Bonatti-D\'{\i}az' result \cite[Theorem 2.1]{BD08}.
%We hereafter suppose $\uind(P)=\uind(Q) +1$ for saddle periodic points $P$ and $Q$ of $f$.
%To show it,
%we first recall the next key result presented by Bonatti-D\'{\i}az.
%A \emph{co-index 1 cycle}  means a heterodimensional cycle 
%with saddle periodic points $P$ and $Q$ whose indices 
%satisfy the co-index 1 condition, i.e.\ $\uind(P)=\uind(Q)\pm1$.  

\begin{lemma}\label{lem2.2}
Let $f$ be any element of $\mathscr{A}$ 
which has a heterodimensional cycle
associated  with $P$ and $Q$ at
which both the derivatives have non-real eigenvalues.
Then every $C^{1}$ neighborhood of $f$ contains a diffeomorphism $f_1$ with a 
heterodimensional cycle 
having real central eigenvalues. 
Moreover, this cycle for $f_1$ can be taken associated with saddles $P^{\prime}_{f_1}$ and $Q^{\prime}_{f_1}$ which are
homoclinically related to the continuations $P_{f_1}$ of $P$ and $Q_{f_1}$ of $Q$, respectively.
\hfill$\square$
\end{lemma}
\noindent
Here the heterodimensional cycle for $f_1
$ is said to have \emph{real central eigenvalues} if
there are 
an expanding real eigenvalue  of $Df_1^{\per(f_1)}(P_{f_1})$ and 
a contracting real eigenvalue  of $Df_1^{\per(f_1)}(Q_{f_1})$ whose multiplicities are equal to $1$.
Note that this lemma implies that $\uind(P)=\uind(P_{f_1})$ and $\uind(Q)=\uind(Q_{f_1})$. 
See  Figure \ref{fg1}
for the configuration of each ingredient in Lemma \ref{lem2.2}.

\begin{remark}\label{rmk2.3}
\normalfont 
The 
heterodimensional cycle 
for $f_1$ in Lemma \ref{lem2.2} 
is \emph{simple} in the sense 
that the local dynamics in small neighborhoods of $P_{f_1}$ and of $Q_{f_1}$ are linear 
while the transitions between the neighborhoods are affine and preserve a partially hyperbolic splitting.
See  \cite[Definition 3.4]{BD08} for the precise description.
Moreover the cycle 
contains a transverse intersection associated with the saddles $P^{\prime}_{f_1}$ and $Q^{\prime}_{f_1}$, see \cite[\S 5]{BD08}.
\end{remark}

We here introduce a concept of intrinsic tangencies. 
Let $f$ be a diffeomorphism on a three-dimensional manifold $M$
having saddle periodic points $P$ and $Q$ with 
$\uind(P)=\uind(Q) +1$.
Suppose that all eigenvalues of 
the derivative $Df^{\per(Q)}(Q)$ are real. 
We say that $W^{u}(P)$ and $W^{s}(Q)$ have an \emph{intrinsic tangency} if 
there is a leaf $\ell^{ss}$ of the $C^{1}$ strong stable foliation $\mathcal{F}^{ss}(Q)$ in $W^{s}(Q)$ 
such that $\ell^{ss}$ and $W^{u}(P)$  have a tangency. 
See Figure \ref{fg0}-(a).
Note that the intrinsic tangency is not necessarily contained in a 
heterodimensional tangency between $W^{u}(P)$ and $W^{s}(Q)$ 
as shown in Figure \ref{fg0}-(b). See  \cite{DR92, KS12} for its precise definition.
Indeed, it is not difficult to give examples where intrinsic tangencies are contained in 
$W^{u}(P)\pitchfork W^{s}(Q)$,  \emph{e.g.},   
circular transverse heterodimensional intersections in \cite{DR92, KS12} contain at least two such intrinsic tangencies.
 %%%%%%%%%%%%%%%%%%%%%%%%%%%%%%%%%%%%%%%
\begin{figure}[hbt] 
\centering
\scalebox{0.8}{\includegraphics[clip]{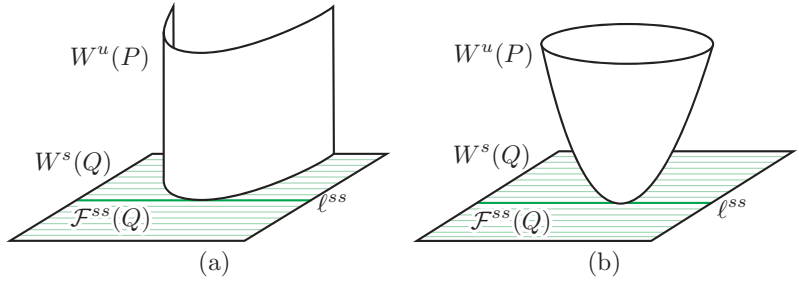}}
\caption{Intrinsic tangencies in transverse intersection  and heterodimensional tangency}
\label{fg0}
\end{figure}
%%%%%%%%%%%%%%%%%%%%%%%%%%%%%%%%%%%%%%%

Let $f$ be any element of $\mathscr{A}$ 
which has a heterodimensional cycle
associated with saddle periodic points $P$ and $Q$ at
which both the derivatives have non-real eigenvalues.
One may suppose that 
$\uind(P)>  \uind(Q)$, 
$W^{u}(P)\pitchfork W^{s}(Q)\neq \emptyset$ and $W^{u}(Q)\cap W^{s}(P)$ 
contains a quasi-transverse intersection.
By Lemma \ref{lem2.2} and Remark \ref{rmk2.3},  one obtains
 a $C^{1}$ diffeomorphism $f_{1}$ arbitrarily $C^{1}$-close to $f$ which satisfies the following properties: 
\begin{enumerate}
 \renewcommand{\theenumi}{A\arabic{enumi}}
\renewcommand{\labelenumi}{(\theenumi)}
\item\label{1f1}  
$f_{1}$ has saddle periodic points $P^{\prime}_{f_{1}}$ and $Q^{\prime}_{f_{1}}$ with the following conditions:
 \begin{enumerate}[(i)]
\item  
the eigenvalues of the derivatives of $f_{1}$ at $P^{\prime}_{f_{1}}$ and $Q^{\prime}_{f_{1}}$ 
are distinct real numbers;
\item $P^{\prime}_{f_{1}}$ is homoclinically related to the continuation $P_{f_{1}}$ of $P$, while 
$Q^{\prime}_{f_{1}}$ is homoclinically related to the continuation $Q_{f_{1}}$ of $Q$
%, i.e., 
%$W^u(P_{f_1}) \pitchfork W^s(P'_{f_1}) \neq \emptyset$, $W^s(P_{f_1}) \pitchfork W^u(P'_{f_1}) \neq \emptyset$, 
%while $W^u(Q_{f_1}) \pitchfork W^s(Q'_{f_1}) \neq \emptyset$, $W^s(Q_{f_1}) \pitchfork W^u(Q'_{f_1}) \neq \emptyset$
;
\item
 $W^u(P_{f_1}) \cap W^s(Q_{f_1})$ contains a  transverse intersection;
\end{enumerate}
\item\label{2f1}  
$f_{1}$ has a  heterodimensional cycle associated with  $P^{\prime}_{f_{1}}$ and $Q^{\prime}_{f_{1}}$,  
 i.e.,  
 \begin{enumerate}[(i)]
\item $W^{u}(P^{\prime}_{f_{1}})\cap W^{s}(Q^{\prime}_{f_{1}})$ contains 
a  transverse intersection; 
\item
 $W^{s}(P^{\prime}_{f_{1}})\cap W^{u}(Q^{\prime}_{f_{1}})$ 
contains a quasi-transverse intersection $X^{\prime}$.
\end{enumerate}
\end{enumerate}
Here $X^{\prime}$ is called a \emph{quasi-transverse intersection} 
if it satisfies 
$T_{X^{\prime}}W^{s}(P^{\prime}_{f_{1}})+T_{X^{\prime}}W^{u}(Q^{\prime}_{f_{1}})
=T_{X^{\prime}}W^{s}(P^{\prime}_{f_{1}})\oplus T_{X^{\prime}}W^{u}(Q^{\prime}_{f_{1}})$.
See Figure \ref{fg1}. 
%%%%%%%%%%%%%%%%%%%%%%%%%%%%%%%%%%%%%%%
\begin{figure}[hbt] 
\centering
\scalebox{0.85}{\includegraphics[clip]{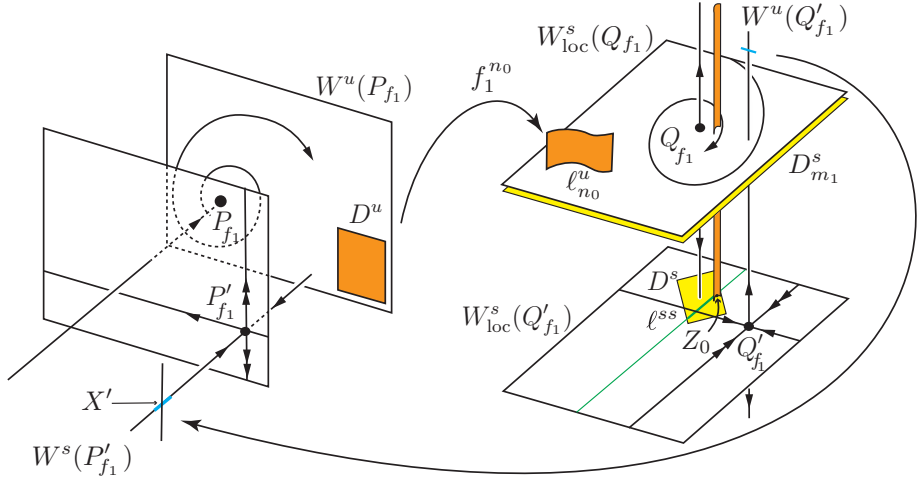}}
\caption{Heterodimensional and equidimensional cycles for $f_{1}$ }
\label{fg1}
\end{figure}
%%%%%%%%%%%%%%%%%%%%%%%%%%%%%%%%%%%%%%%

\begin{lemma}\label{lem:p5c}
Arbitrarily $C^{1}$-close to $f_{1}$ satisfying (\ref{1f1}) and  (\ref{2f1}), there is a $C^{r}$ diffeomorphism $f_{2}$  satisfying (\ref{1f1}) and  (\ref{2f1}) such that 
$W^u(P_{f_2})\pitchfork W_{\loc}^s(Q^{\prime}_{f_2})$ contains an intrinsic tangency, where 
$P_{f_2}$ and $Q^{\prime}_{f_2}$ are the continuations of $P_{f_1}$ and $Q^{\prime}_{f_1}$, respectively.
\end{lemma}
\noindent
Together with Lemma \ref{lem2.2} and Remark \ref{rmk2.3}, it immediately implies  the next result.
\begin{corollary}\label{coro:p5c}
The intrinsic tangency is obtained by an arbitrarily small $C^1$-perturbation of any $f\in \mathscr{A}$.
\hfill$\square$
\end{corollary}
\begin{proof}[Proof of Lemma \ref{lem:p5c}]
We here recall that the above $f\in \mathscr{A}$ has the saddle periodic point $Q$ such that 
$Df^{\per(Q)}(Q)$ has a pair of non-real contracting eigenvalues. 
Hence, there exist 
a  small neighborhood $V$ of $Q_{f_1}$,  
a local chart $(x,y,z)$ in $V$
and 
real constants $a_s, a_u, \vartheta \in \mathbb R$ with 
$0<\vert a_s\vert <1 <\vert a_u\vert$ 
such that
\begin{itemize}
\item $\overline{V}\cap W^{u}_{\loc}(P_{f_{1}})=\emptyset$ 
and $\overline{V}\cap \mathcal{O}_{f_{1}}(X^{\prime})=\emptyset$, where 
$\overline{V}$ is the closure of $V$ and 
$\mathcal{O}_{f_{1}}(\cdot)$ is the 
orbit of the corresponding point for $f_{1}$; 
\item 
$Q_{f_1} =(0,0,0)$ and 
\begin{equation}\label{eq:p4curve3}
f_1^{\per(Q_{f_1})}(x,y,z) =\left( (x,y)\; {}^t\!A_s,\ a_uz \right), \  A_s=a_s\left( \begin{array}{cc} \cos 2\pi \vartheta & -\sin 2\pi \vartheta\\ \sin 2\pi \vartheta & \cos 2\pi \vartheta \end{array} \right)
\end{equation}
for any $(x,y,z) \in V$. 
Moreover, after a small local perturbation if necessary, we may assume that the above $\vartheta$ is irrational. 
\end{itemize}

On the one hand, by (\ref{1f1})-(iii), 
there exist an unstable disk 
$$D^{u}\subset W^{u}_{\loc}(P_{f_{1}})\cap V^{c}$$  and 
a positive integer $n_{0}$  such that the set  of transverse intersections 
$f_{1}^{n_{0}}(D^{u})\pitchfork W^{s}_{\loc}(Q_{f_1})$ 
contains an arc $\ell^{u}_{n_{0}}$. See Figure \ref{fg1}. 
On the other hand, by (\ref{1f1})-(ii) and the Inclination Lemma, 
it follows 
that $W^{s}(Q^{\prime}_{f_{1}})$ contains two-dimensional disks
for which the backward images converge to  $W_{\loc}^{s}(Q_{f_{1}})$ in $C^{1}$ topology. 
Hence, 
for any $\epsilon>0$, 
there exist an integer $m_{0}\geq 0$ and a stable disk 
$D^{s}\subset W^{s}_{\loc}(Q^{\prime}_{f_{1}})$
containing a point of $W^{u}(Q_{f_{1}})\pitchfork W^{s}_{\loc}(Q^{\prime}_{f_{1}})$ and 
such that, for any integer $m\geq 0$,
$$
d_{C^1}\left(D^{s}_{m},\ W_{\loc}^{s}(Q_{f_{1}})\right)<\epsilon
$$
where $d_{C^1}(\cdot ,\cdot)$ is the $C^1$ distance between corresponding submanifolds, 
and 
$D^{s}_{m}$ is a component of $f_{1}^{-m\per(Q_{f_{1}})-m_{0}}(D^{s})\cap V$.

Note that, by  Remark \ref{rmk2.3} together with (\ref{1f1})-(i), 
one has the strong stable foliation 
$\mathcal{F}^{ss}(Q^{\prime}_{f_{1}})$ of $Q^{\prime}_{f_{1}}$ 
whose leaves are of codimension one  in $W^{s}(Q^{\prime}_{f_{1}})$. 
Moreover,
observe that, by the irrational rotation of (\ref{eq:p4curve3}) with $\vartheta\not\in \mathbb{Q}$,
the images of $\ell ^u_{n_0}$ by the forward iterations of $f_1^{\per(Q_1)}$
rotate and converge
to $Q_{f_1}$ in $W^{s}_{\loc}(Q_{f_1})$, see Figure \ref{fg1-2}-(a).
Hence, 
there are 
an integer $m_{1}\geq 0$
and 
an arc $\ell^{ss}\subset D^{s}\cap \mathcal{F}^{ss}_{\loc}(Q^{\prime}_{f_{1}})$ 
such that 
\begin{itemize}
\item  $\ell^{u}_{n_{0}}\cap \ell^{ss}_{m_{1}}\neq \emptyset$ where 
$\ell^{ss}_{m_{1}}:=\pi^{s}\left(f_{1}^{-m_{1}\per(Q_{f_{1}})-m_{0}}(\ell^{ss})\right)$ 
for the canonical projection $\pi^{s}:V\to W^{s}_{\loc}(Q_{f_1})$ with $\pi^{s}(x, y, z)=(x, y, 0)$;
\item there is a point ${\bm p}_{0}\in \ell^{u}_{n_{0}}\cap \ell^{ss}_{m_{1}}$ such that 
$$
\angle\left(T_{{\bm p}_{0}}\ell^{u}_{n_{0}},\  T_{{\bm p}_{0}} \ell^{ss}_{m_{1}}\right)<\epsilon,
$$
where $\angle(\cdot,\cdot)$ stands for the angle  
between the corresponding subspaces in $T_{{\bm p}_{0}} W^{s}_{\loc}(Q_{f})$.  
\end{itemize}
%%%%%%%%%%%%%%%%%%%%%%%%%%%%%%%%%%%%%%%
\begin{figure}[hbt] 
\centering
\scalebox{0.85}{\includegraphics[clip]{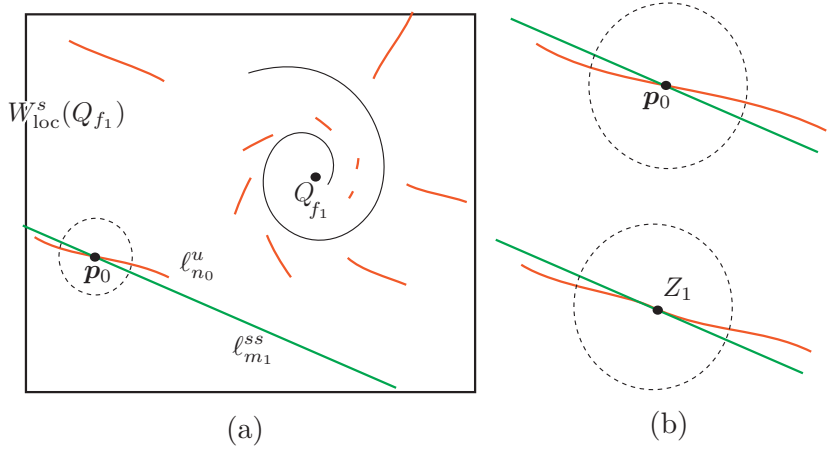}}
\caption{Images of $\ell^{u}_{n_{0}}$ and $\ell^{ss}_{m_{1}}$}
\label{fg1-2}
\end{figure}
%%%%%%%%%%%%%%%%%%%%%%%%%%%%%%%%%%%%%%%
Therefore, one can obtain a diffeomorphism $f_2$ with a tangency $Z_{1}$ near  ${\bm p}_{0}$ between $\ell^{u}_{n_{0}}$ and $\ell^{ss}_{m_{1}}$,
if necessary perturbing $f_{1}^{n_{0}}$ slightly 
in a small neighborhood of $f_{1}^{-n_{0}}({\bm p}_{0})$ in $D^{u}$. 
It follows from the above conditions
 that such a $C^r$ perturbation deforming $\ell^{u}_{n_{0}}$ as shown in Figure \ref{fg1-2}-(b)
can be defined by the composition of appropriate bump function and isometry.
Thus we have 
the intrinsic tangency 
$$Z_{0}:=f_{2}^{m_{1}\per(Q_{f_{2}})+m_{0}}(Z_{1})$$
between $W^u(P_{f_2})$ and $W^{s}_{\loc}(Q_{f_{2}}^{\prime})$.
This concludes the proof of Lemma \ref{lem:p5c}.
\end{proof}

\begin{proof}[Proof of Proposition \ref{prop2.1}]
Arbitrarily $C^1$ close to $f\in \mathscr{A}$, 
from Lemmas \ref{lem2.2},  
one has a  diffeomorphism $f_1$ with a 
heterodimensional cycle associated with saddles $P^{\prime}_{f_1}$ and $Q^{\prime}_{f_1}$ 
with distinct real eigenvalues 
which are
homoclinically related to the continuations $P_{f_1}$ and $Q_{f_1}$, respectively.
Moreover, it follows 
by Lemma \ref{lem:p5c} and Corollary \ref{coro:p5c} that, arbitrarily $C^1$ near $f_1$, 
one obtains  a  $C^r$ diffeomorphism $f_2$ 
satisfying (\ref{1f1}) and  (\ref{2f1}) which has an intrinsic tangency $Z_0$ between 
$W^u(P_{f_2})$ and 
 a  leaf $\ell^{ss}$ of the strong stable foliation $\mathcal{F}^{ss}(Q^{\prime}_{f_{2}})$.
Note that $\ell^{ss}$ is almost parallel to $W_{\loc}^{ss}(Q^{\prime}_{f_{2}})$ 
in the linearizing coordinates in a neighborhood of $Q^{\prime}_{f_{2}}$. 
See Figure \ref{fg2}-(a).

Recall that, by (\ref{2f1}), 
$W^{s}_{\loc}(P^{\prime}_{f_{2}})\cap W^{u}(Q^{\prime}_{f_{2}})$ 
contains the quasi-transverse intersection $X^{\prime}$
as shown in Figure \ref{fg1}. Hence 
one has a positive integer $k_{0}$ such that 
the point $X^{\prime}_0=f_{2}^{k_{0}}(X^{\prime})$ is contained in $W^{u}_{\loc}(Q_{f_{2}}^{\prime})$.
Note that 
since  $\overline{V}\cap \mathcal{O}_{f_{2}}(X^{\prime})=\emptyset$ where $V$ is the neighborhood $V$ of $Q_{f_1}$, 
$X^{\prime}_0\not \in \overline{V}$.
We here
consider a small segment $L^{s}\subset W^{s}(P^{\prime}_{f_{2}})$ with 
 $X^{\prime}_0\in L^{s}$ and $L^{s}\cap \overline{V}=\emptyset$. 
 Observe that 
if necessary perturbing $f_{2}$ slightly in a small neighborhood of $L^{s}$, 
it follows from Inclination Lemma that, 
for any  large  integer $m>0$, 
the backward image $f_{2}^{-m\per(Q^{\prime}_{f_2})}(L^{s})$ 
contains a segment 
which  is sufficiently $C^{1}$-close to  $W_{\loc}^{ss}(Q^{\prime}_{f_{2}})$. 
Denote $f_{2}^{-m\per(Q^{\prime}_{f_2})}(X^{\prime}_{0})$ by $X^{\prime}_{m}$. 
See Figure \ref{fg2}-(b).

 %%%%%%%%%%%%%%%%%%%%%%%%%%%%%%%%%%%%%%%
\begin{figure}[hbt] 
\centering
\scalebox{0.85}{\includegraphics[clip]{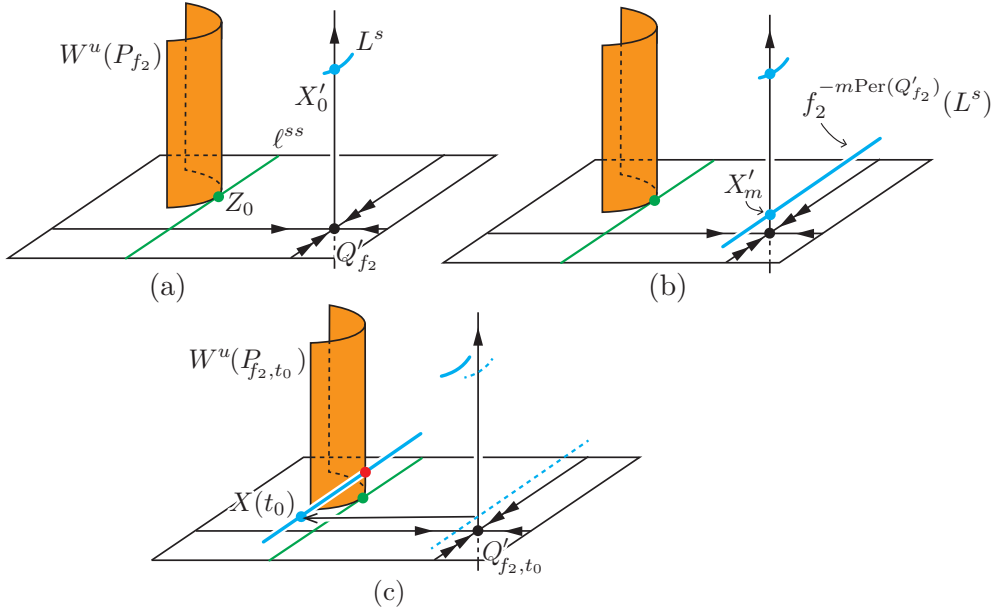}}
\caption{The appearance of an intrinsic tangency}
\label{fg2}
\end{figure}
%%%%%%%%%%%%%%%%%%%%%%%%%%%%%%%%%%%%%%%

We here consider a one-parameter family of $C^{r}$ diffeomorphisms, $\{ f_{2, t}\}_{t\in \mathbb{R}}$  with $f_{2,0}=f_{2}$, 
which \emph{unfolds generically} the quasi-transverse intersection at $X^{\prime}_{m}$ used in \cite{DR92, KNS10}, 
that is, there are  a $C^{1}$-curve $\{X(t)\}_{t\in \mathbb{R}}$ and 
a $C^{1}$ map $\rho : \mathbb{R}\to \mathbb{R}^{+}$ with 
$X(t)\in f_{2, t}^{k_{0}+m}(W_{\loc}^{u}(P_{f_{2, t}}))$ for every $t$ and $\rho(0)>0$ such that:
\begin{itemize}
\item $X(0)=X^{\prime}_{m}$ and $\mathrm{dist}\left(X(t), W^{u}_{\loc}(Q^{\prime}_{f_{2, t}}) \right)=|t| \rho(t)$;
\item 
$T_{X^{\prime}_{m}} W^{s}(P^{\prime}_{f_{2, t}}) \oplus T_{X^{\prime}_{m}} W^{u}(Q^{\prime}_{f_{2, t}})  \oplus N 
= T_{X^{\prime}_{m}}  M$, where $N$ is the one-dimensional space spanned by $\frac{d X}{dt}(0)$.
\end{itemize}
Observe that, there is a real number $t_{0}$ arbitrarily near $0$
such that 
$f_{2, t_{0}}^{-m\per(Q^{\prime}_{f_{2, t_0}})}(L^s)$ and $W^{u} (P_{f_{2, t_{0}}})$ have a quadratic tangency. See Figure \ref{fg2}-(c). 
Let us denote $f_{2, t_{0}}$ by $g$. 
It follows  that $g$ satisfies the condition (\ref{B3}). 

Note that it follows from (\ref{1f1})-(ii) that  (\ref{B2}) holds for $g$. 
In addition, 
since the amount of all perturbations can be taken arbitrarily small, 
 (\ref{B1}) holds for $g$. In conclusion,  $g$ is contained in $\mathscr{B}$. 
 This ends the proof of Proposition \ref{prop2.1}.
\end{proof}
\noindent
The above proof contains the following remark: 
\begin{remark}
Every $C^r$ diffeomorphism satisfying (\ref{1f1}) and (\ref{2f1}) 
can be $C^r$ approximated by diffeomorphisms with a non-transverse equidimensional cycle satisfying the conditions (\ref{B1}), (\ref{B2}) and (\ref{B3}).
\end{remark}

\section{Homoclinic tangencies with the Tatjer condition}\label{sec3}
We first recall a diffeomorphism with a generalized homoclinic tangency satisfying some 
conditions 
presented in \cite{Tj01}, which plays an important role in 
the proof of Theorem \ref{thm1}. 
The purpose of this section is to show  Proposition \ref{thm3.1}, 
where we claim that diffeomorphisms satisfying the Tatjer condition
are \emph{not} so special in a neighborhood of $\mathscr{B}$. 

Let $f$ be a diffeomorphism on a  three-dimensional Riemannian manifold $M$
which has a homoclinic tangency of a saddle periodic point $P^{\prime}$.
We here note that one of the requirement in Tajter's conditions 
is that the central-stable bundle $E^{cs} = E^s \oplus E^c$ at $P^\prime$ 
must be extended along the stable manifold $W^s(P^\prime)$ of $P^\prime$. 
See the explanation just below (\ref{C2}).
To guarantee it as well as the quadratic condition in  (\ref{C1}), 
we have to assume that the regularity of $f$ is at least $C^2$.
Suppose 
the derivative for $f$ at $P^\prime$ has  real eigenvalues $\lambda_{s}, \lambda_{cu}$ and $\lambda_{u}$ 
satisfying $|\lambda_{s}|<1<|\lambda_{cu}|<|\lambda_{u}|$. 
Assume that there are $C^{1}$ linearizing coordinates $(x, y, z)$ for $f$ on a neighborhood $U^{\prime}$ of $P^{\prime}$ such that 
$$P^{\prime}=(0,0,0),\quad 
f(x,y,z)=(\lambda_{cu} x,\ \lambda_{u} y,\ \lambda_{s}z)$$
for any $(x, y, z)\in U$. 
 In $U^{\prime}$, the local  unstable and stable manifolds of $P^{\prime}$ are  given respectively as
$$
W_{\loc}^{u}(P^{\prime})=\left\{(x, y, 0)\ ;\  |x|, |y|<\delta\right\},\ 
W_{\loc}^{s}(P^{\prime})=\left\{(0, 0, z)\ ;\  |z|<\delta\right\}
$$ for some $\delta>0$. 
Moreover one has the local strong unstable $C^{1}$ foliation $\mathcal{F}^{uu}_{\loc}(P^{\prime})$ 
in $W_{\loc}^{u}(P^{\prime})$ 
such that, 
for  any point $\bar{\boldsymbol{x}}=(\bar{x}, \bar{y}, 0)\in W_{\loc}^{u}(P^{\prime})$, 
the leaf $\ell^{uu}(\bar{\boldsymbol{x}})$ 
of $\mathcal{F}^{uu}_{\loc}(P^{\prime})$ containing $\bar{\boldsymbol{x}}$ 
is given as 
$$
\ell^{uu}(\bar{\boldsymbol{x}})=
\left\{(\bar{x}, y, 0)\ ;\ |y|<\delta \right\}\!.
$$ 

%For any point $\bar{\boldsymbol{x}}=(\bar{x}, 0, 0)\in W^{u}_{\loc}(P^{\prime})$, 
%the \emph{s-fiber} $\Phi^{s}(\bar{\boldsymbol{x}})$  at $\bar{\boldsymbol{x}}$ 
%is defined as the tangent space generated by the vectors 
%$(\frac{\partial}{\partial x})_{\bar{\boldsymbol{x}}}$, 
%$(\frac{\partial}{\partial y})_{\bar{\boldsymbol{x}}}\in T_{\bar{\boldsymbol{x}}} M$. 
%Furthermore, 
%if a point $\boldsymbol{x}\in W^{s}(P^{\prime})$ satisfies  
%$\bar{\boldsymbol{x}}=f^{-n}(\boldsymbol{x})\in W_{\loc}^{s}(P^{\prime})$
%for some integer $n>0$,
%we define  
%$$ \Phi^{s}(\boldsymbol{x}):=Df^{n}(\bar{\boldsymbol{x}})( \Phi^{s}(\bar{\boldsymbol{x}}) ),$$
%which is called the \emph{s-fiber}  at $\boldsymbol{x}$.

We say that a homoclinic tangency satisfies the \emph{Tatjer condition} 
(which corresponds to the type I of case B in \cite[Theorem 1]{Tj01}) 
if the following (\ref{C1})--(\ref{C3}) hold. 
\begin{enumerate}
\renewcommand{\theenumi}{C\arabic{enumi}}
\renewcommand{\labelenumi}{(\theenumi)}
\item \label{C1}
$W^{u}(P^{\prime})$ and $W^{s}(P^{\prime})$ have a quadratic tangency at 
$\boldsymbol{x}_{0}$ which does not
belong to  the strong unstable manifold $W^{uu}(P^{\prime})$ of $P^{\prime}$;
\item \label{C2}
$W^{s}(P^{\prime})$ is tangent  to the leaf $\ell^{uu}(\boldsymbol{x}_{0})$  of  $\mathcal{F}_{\loc}^{uu}(P^{\prime})$ at $\boldsymbol{x}_{0}$.
\end{enumerate}
For the tangency $\boldsymbol{x}_{0}$, 
we here consider the forward image $\bar{\boldsymbol{x}}_{0}=f^{n_{0}}(\boldsymbol{x}_{0})$ 
for a large $n_{0}\geq 0$
satisfying 
$\bar{\boldsymbol{x}}_{0}\in W^{s}_{\loc}(P^{\prime})$. 
In addition, we consider 
a plane $S(\bar{\boldsymbol{x}}_{0})$ 
containing $\bar{\boldsymbol{x}}_{0}$ such that 
$T_{\bar{\boldsymbol{x}}_{0}} S(\bar{\boldsymbol{x}}_{0})$ is generated by  
$(\frac{\partial}{\partial x})_{\bar{\boldsymbol{x}}_{0}}$, 
$(\frac{\partial}{\partial z})_{\bar{\boldsymbol{x}}_{0}}\in T_{\bar{\boldsymbol{x}}_{0}} M$. 
Note that by the chosen linearizing coordinates on a neighborhood of $P^\prime$, the plane $S(\bar{\boldsymbol{x}}_{0})$ in $T_{\bar{\boldsymbol{x}}_{0}} M$ corresponds to the central-stable bundle at this tangent point.
See Figure \ref{fg3}.
The last condition is the following: 
\begin{enumerate}
\setcounter{enumi}{2}
\renewcommand{\theenumi}{C\arabic{enumi}}
\renewcommand{\labelenumi}{(\theenumi)}
\item \label{C3}
$S(\bar{\boldsymbol{x}}_{0})$ 
and 
$W^{u}(P^{\prime})$ 
are transverse at $\bar{\boldsymbol{x}}_{0}$.
\end{enumerate}
%%%%%%%%%%%%%%%%%%%%%%%%%%%%%%%%%%%%%%%
\begin{figure}[hbt] 
\centering
\scalebox{0.86}{\includegraphics[clip]{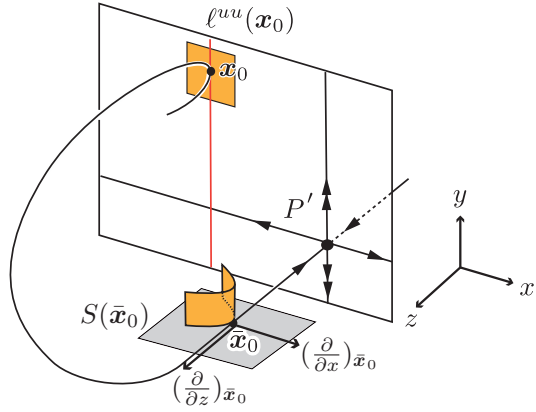}}
\caption{The Tatjer condition}
\label{fg3}
\end{figure}
%%%%%%%%%%%%%%%%%%%%%%%%%%%%%%%%%%%%%%%

Denote by $\mathscr{C}$ the set of all $C^{r}$   
diffeomorphisms on  $M$
which have homoclinic tangencies satisfying the Tatjer condition.
This set $\mathscr{C}$ is contained in the class of diffeomorphisms having ``generalized homoclinic tangencies'' defined in \cite{Tj01}.
\smallskip 

We here claim that homoclinic tangencies satisfying the Tatjer condition
are \emph{not} so rare in our context. 
Let $\ell_{1}$ and $\ell_{2}$ be submanifolds of $M$. 
For an intersection $\boldsymbol{x}\in \ell_{1}\cap \ell_{2}$, define 
$$c_{\boldsymbol{x}}( \ell_{1}, \ell_{2})=
\dim M-\bigr(\dim T_{\boldsymbol{x}}\ell_{1}+\dim T_{\boldsymbol{x}}\ell_{2}-\dim (T_{\boldsymbol{x}}\ell_{1}\cap T_{\boldsymbol{x}}\ell_{2})\bigr),
$$ which is called the 
\emph{codimension} at the intersection $\boldsymbol{x}$ associated with  $\ell_{1}$ and $\ell_{2}$. 
See \cite{BR-ax}.
The condition (\ref{C3}) implies that 
the codimension at the intersection $\bar{\boldsymbol{x}}_{0}$
associated with $S(\bar{\boldsymbol{x}}_{0})$ and $W^{u}(P^{\prime})$ 
is zero. 
However, 
(\ref{C1}) and (\ref{C2}) require pairs of submanifolds of codimension one as well as two:
$$c_{\boldsymbol{x}_{0}}( W^{u}(P^{\prime}), W^{s}(P^{\prime}))=1, \quad
c_{\boldsymbol{x}_{0}}(  W^{s}(P^{\prime}), \ell^{uu}(\boldsymbol{x}_{0}))=2.$$
Thus, the homoclinic tangency with the Tatjer condition seems to be very special, as mentioned in \cite{Tj01}.
However, 
it can be realized by $C^{r}$-small perturbations of any elements of $\mathscr{B}$. That is,

\begin{proposition}\label{thm3.1}
 $\mathscr{B}$ is contained in the $C^{r}$, $r\geq 2$, closure of  $\mathscr{C}$. 
\end{proposition}
\begin{proof}
Let us consider $f\in \mathscr{B}$ having
saddle periodic points  $P_{f}$ and $P^{\prime}_{f}$ and 
satisfying (\ref{B1})--(\ref{B3}). 
Since the unstable eigenvalues at $P_{f}$ are non-real, 
if necessary perturbing $f$ slightly in a small neighborhood of $P_{f}$ without breaking (\ref{B1})--(\ref{B3}), 
we may assume from the beginning that 
there exist a local chart $(\bar{x}, \bar{y}, \bar{z})$ in a small neighborhood $U$ of $P_{f}$
and  real constants
$ b_{s}, b_{u}, \vartheta$ such that 
\begin{enumerate}
\item  $0<|b_{s}|<1<|b_{u}|$ and $\vartheta\in [0, 1]$ which is an irrational number; 
\item 
$P_{f}=(0,0,0)$ and for any $(\bar{x}, \bar{y}, \bar{z})\in U$,   
 \begin{equation} \label{eq3.1.1}
 f^{\per(P_{f})}(\bar{x}, \bar{y}, \bar{z})=
 \left(
(\bar{x}, \bar{y})\;  {}^t\!B_{u},\  b_{s} \bar{z}
 \right),
\ 
B_{u}= b_{u}\left(\begin{array}{cc}
\cos 2\pi\vartheta & -\sin 2\pi\vartheta \\
\sin 2\pi\vartheta & \cos 2\pi\vartheta \\
\end{array}\right).
\end{equation}
\end{enumerate}

%\begin{equation} \label{eq3.1.1} 
%Df^{\per(P_{f})}(P_{f})
%=\left(\begin{array}{ccc}
%b_{u}\cos 2\pi\vartheta & -b_{u}\sin 2\pi\vartheta  & 0\\
%b_{u}\sin 2\pi\vartheta & b_{u}\cos 2\pi\vartheta & 0\\
%0 & 0 & b_{s}
%\end{array}\right).
%\end{equation}

Let  $\bar{\boldsymbol{x}}_{0}=(0, 0, \bar{z}_{0})$ 
be a point in 
$W^{u}(P_{f}^{\prime})\pitchfork W^{s}_{\loc}(P_{f})$ given  by (\ref{B2}) as shown in Figure \ref{fg4}-(a). Without loss of generality
we may suppose that $\bar{\boldsymbol{x}}_{0}\not\in W^{uu} (P_{f}^{\prime})$.  
Define 
$$\bar{\boldsymbol{x}}_{n}:=f^{n\per(P_{f})} (\bar{\boldsymbol{x}}_{0})$$ 
for any  integer $n>0$, see Figure \ref{fg4}-(b).
%%%%%%%%%%%%%%%%%%%%%%%%%%%%%%%%%%%%%%%
\begin{figure}[hbt] 
\centering
\scalebox{0.86}{\includegraphics[clip]{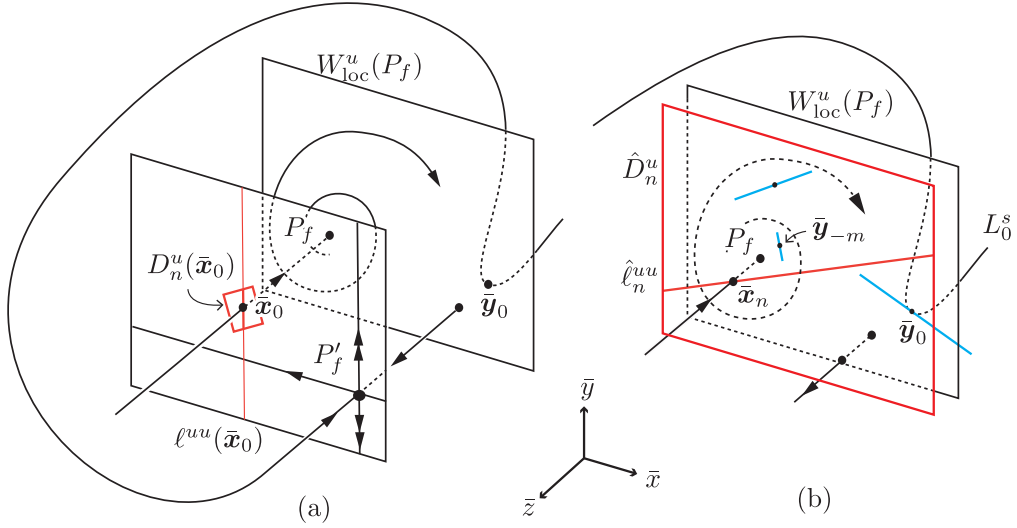}}
\caption{Non-transverse equidimensional cycle with rotation}
\label{fg4}
\end{figure}
%%%%%%%%%%%%%%%%%%%%%%%%%%%%%%%%%%%%%%%
By the Inclination Lemma,  for a large $n$, there is a two-dimensional disk 
$D_{n}^{u}(\bar{\boldsymbol{x}}_{0})\subset W^{u}(P_{f}^{\prime})$ 
containing $\bar{\boldsymbol{x}}_{0}$ 
such that 
$f^{n\per(P_{f})} (D_{n}^{u}(\bar{\boldsymbol{x}}_{0}) )$ 
converges to 
$W^{u}_{\loc}(P_{f})$ in the $C^{r}$ topology as $n\to \infty$. 
Write
$$\hat{D}_{n}^{u}:=
f^{n\per(P_{f})} (D_{n}^{u}(\bar{\boldsymbol{x}}_{0}) )$$
for each $n>0$. 
Note that 
there is a segment  $\ell^{uu}_{n}(\bar{\boldsymbol{x}}_{0})$ contained in the 
the leaf through $\bar{\boldsymbol{x}}_{0}$
of the strong unstable foliation $\mathcal{F}^{uu}(P^{\prime}_{f})$,  
which is carried to $\hat{D}_{n}^{u}$ by $f^{n \per(P_{f})}$.
For each $n>0$,  define 
$$\hat{\ell}_{n}^{uu}:= f^{n \per(P_{f})}(\ell^{uu}(\bar{\boldsymbol{x}}_{0})).$$
%Let  $L$ be a $C^{1}$ curve contained in $W^{u}_{\loc}(P_{f})$ through $P_{f}$. 
See Figure \ref{fg4}-(b). 
Let us take the following steps:

%In the next step, 
%we estimate the $C^{1}$ distance between $\hat{\ell}_{n}^{uu}$ and $L$ in a small neighborhood of $P_{f}$.
 \smallskip
 
\noindent
\textbf{Step 1:}
Let $\boldsymbol{v}_n^{uu}$ be a unit vector tangent to $\hat{\ell}_n^{uu}$ at $\bar{\boldsymbol{x}}_n$.
Since the sequence $\{\bar{\boldsymbol{x}}_n\}$ converges to $P_f$ in $M$ as $n\to \infty$, 
we have a subsequence $\{\boldsymbol{v}_{n_i}^{uu}\}$ of $\{\boldsymbol{v}_n^{uu}\}$ converging to 
a unit vector $\boldsymbol{v}_\infty\in %T_{P_f}(M)
T_{P_f} W^u_{\loc}(P_f)
$ as $n_i\to \infty$ in the tangent vector space 
$TM$.
This implies that, for any $\varepsilon>0$, there is an integer $n_{0}:=n_{i_{0}}>0$ such that 
\begin{equation}\label{eq3.1.2}
\left|P_{f}- \bar{\boldsymbol{x}}_{n_{0}}\right|<\varepsilon/2,\quad  
\mathrm{dist}_{TM}(\boldsymbol{v}_\infty, \boldsymbol{v}_{n_0}^{uu})<\varepsilon/2.
\end{equation}
%where 
%$\pi^{u}:U\to W^{u}_{\loc}(P_{f})$ is the canonical projection 
%with $\pi^{u}(\bar{x}, \bar{y}, \bar{z})=(\bar{x}, \bar{y},0)$.  
This finishes the first step.
 \hspace*{ \fill}$\blacksquare$
 \smallskip

Next,  we consider a quadratic tangency  $\bar{\boldsymbol{y}}_{0}$ between 
$W^{s}(P_{f}^{\prime})$ and $ W^{u}_{\loc}(P_{f})$ which is given by the condition (\ref{B3}).
For every integer $m\geq 0$, we write 
$$\bar{\boldsymbol{y}}_{-m}:=f^{-m \per(P_{f})}(\bar{\boldsymbol{y}}_{0}).$$
See Figure \ref{fg4}-(b). 
%In the next step, 
%we estimate $C^{1}$ distance between the $L$ above 
%and the tangent space of $W^{s}(P^{\prime}_{f})$ 
%at $\bar{\boldsymbol{y}}_{-m}$ with a large $m$
%in a small neighborhood of $P_{f}$.
 \smallskip
 
\noindent
\textbf{Step 2:} 
Let $L^{s}_{-m}$ be a subarc of $W^{s}(P_{f}^{\prime})$ passing through 
$\bar{\boldsymbol{y}}_{-m}$ for any integer $m\geq 0$. 
Consider a unit vector $\boldsymbol{w}_{-m}^s$ tangent to $L^{s}_{-m}$ 
at $\bar{\boldsymbol{y}}_{-m}$.
Since $|b_u|>1$ in (\ref{eq3.1.1}), $\{\bar{\boldsymbol{y}}_{-m}\}$ converges to 
$P_f$ in $M$ as $m\to \infty$.
Since moreover $\vartheta$ is irrational, there exists a subsequence $\{m_i\}$ of $\{m\}$ such that 
$\{\boldsymbol{w}_{-m_i}^s\}$ converges to $\boldsymbol{v}_\infty$ in $TM$.
This implies that, for any $\varepsilon>0$, 
there is an integer $m_{0}=m_{i_{0}}>0$ such that 
\begin{equation}\label{eq3.1.3}
\left|P_{f}- \bar{\boldsymbol{y}}_{-m_{0}}\right|<\varepsilon/2,
\quad  
\mathrm{dist}_{TM}(\boldsymbol{v}_\infty, \boldsymbol{w}_{-m_0}^s)<\varepsilon/2.
\end{equation}
This ends the second step.
 \hspace*{ \fill}$\blacksquare$

In the next final step, we combine the above two steps.

\noindent
\textbf{Step 3:}
From (\ref{eq3.1.2}) and (\ref{eq3.1.3}), we have
\begin{equation}\label{eq3.1.4}
\left|\bar{\boldsymbol{x}}_{n_{0}}- \bar{\boldsymbol{y}}_{-m_{0}}\right|<\varepsilon,
\quad 
\mathrm{dist}_{TM}(\boldsymbol{v}_{n_0}^{uu}, \boldsymbol{w}_{-m_0}^s)<\varepsilon.
\end{equation}
See Figure \ref{fg5}-(a). 
This concludes the third step.
 \hspace*{ \fill}$\blacksquare$
\smallskip

%%%%%%%%%%%%%%%%%%%%%%%%%%%%%%%%%%%%%%%
\begin{figure}[hbt] 
\centering
\scalebox{0.88}{\includegraphics[clip]{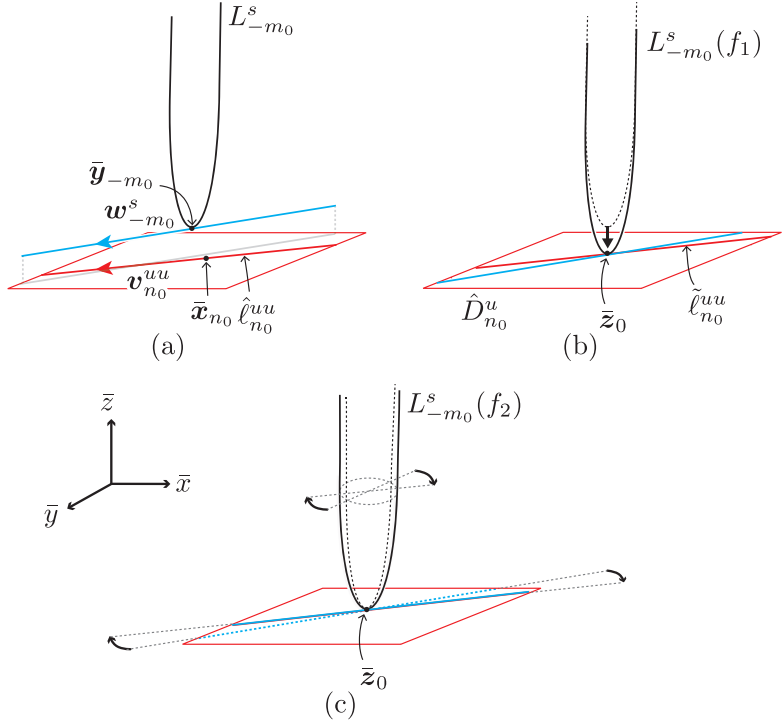}}
\caption{The advent of tangency of codimension two }
\label{fg5}
\end{figure}
%%%%%%%%%%%%%%%%%%%%%%%%%%%%%%%%%%%%%%%

Now we add a small perturbation to $f$ on a  neighborhood of $\bar{\boldsymbol{y}}_{0}$ 
to obtain a $C^r$ diffeomorphism $f_{1}$ which is $C^{r}$-close to $f$ and 
such that 
the continuation $L^{s}_{-m_{0}}(f_{1})$ is obtained from $L^{s}_{-m_{0}}$ 
by a ``shifting down'' operation along the $\bar{z}$-axis and has a quadratic tangency $\bar{\boldsymbol{z}}_{0}$ with $\hat{D}_{n_{0}}^{u}$. 
See Figure  \ref{fg5}-(b).
Let $\tilde{\ell}_{n_0}^{uu}$ be a curve in $\hat D_{n_0}^u$ passing through $\bar{\boldsymbol{z}}_{0}$ 
and such that $f_1^{-n_0 \mathrm{Per}(P_{f_1})}(\tilde{\ell}_{n_0}^{uu})$ is 
contained in one of leaves of $\mathcal{F}^{uu}(P_{f_1}')$.
In general, 
$T_{\bar{\boldsymbol{z}}_{0}} L^{s}_{-m_{0}}(f_{1})$ does not coincide with $T_{\bar{\boldsymbol{z}}_{0}} \tilde{\ell}^{uu}_{n_{0}}$.
However, the condition (\ref{eq3.1.4}) implies that these spaces are sufficiently close to each other. 
Hence, by adding a perturbation again to $f_{1}$ on the neighborhood of $\bar{\boldsymbol{y}}_{0}$, we 
obtain a $C^r$ diffeomorphism $f_{2}$ which is $C^{r}$-close to $f_{1}$ 
and such that the continuation $L^{s}_{-m_{0}}(f_{2})$ of $L^{s}_{-m_{0}}(f_{1})$ has a quadratic 
tangency with $\hat{D}_{n_0}^u$ at $\bar{\boldsymbol{z}}_{0}$ and satisfies
$$T_{\bar{\boldsymbol{z}}_{0}} L^{s}_{-m_{0}}(f_{2})=T_{\bar{\boldsymbol{z}}_{0}} \tilde{\ell}^{uu}_{n_{0}},$$
see Figure \ref{fg5}-(c).
More precisely, $L^{s}_{-m_{0}}(f_{2})$ is obtained from $L^{s}_{-m_{0}}(f_{1})$ by a small rotation 
around the axis meeting $\hat D_{n_0}^u$ orthogonally at $\bar{\boldsymbol{z}}_{0}$.

At last we have obtained 
the $C^{r}$  diffeomorphism $g:=f_{2}$ which is  $C^{r}$-near $f$  and  such that 
\begin{itemize}
\item 
$W^{s}(P^{\prime}_{g})$ and 
$\hat{D}_{n_{i_{0}}}^{u}$
have a quadratic tangency at $\bar{\boldsymbol{z}}_{0}$;
\item 
moreover, at the point $\bar{\boldsymbol{z}}_{0}$, 
$W^{s}(P_{g}^{\prime})$ and  $\tilde{\ell}^{uu}_{n_{0}}$ 
have a tangency with 
$$
c_{\bar{\boldsymbol{z}}_{0}}(W^{s}(P_{g}^{\prime}),  \tilde{\ell}^{uu}_{n_{0}})=2.
$$ 
\end{itemize}
This implies the condition (\ref{C2}). 
Note that, since $\bar{\boldsymbol{x}}_{0}\not\in W^{uu} (P_{f}^{\prime})$, 
the quadratic tangency  $\bar{\boldsymbol{z}}_{0}$
does not belong to  $W^{uu} (P_{f}^{\prime})$. 
It implies that  (\ref{C1}) holds for $g$. 
Moreover, since (\ref{C3}) is the condition about the transversality, it still holds
after arbitrarily small perturbations of $f$. 
Consequently, $g$ belongs to $\mathscr{C}$. 
This completes the proof of Proposition \ref{thm3.1}. 
\end{proof}

In the end of this section,  
we present the next lemma which is indispensable for our discussion in the final section.
Since this is just an extract from the main result of \cite{Tj01}, 
we here skip the proof. 

\begin{lemma}[see the items 1--3 in {\cite[Theorem 1]{Tj01}}]\label{lem3.2}
Let $f$ be a $C^{r}$ $(r\geq 2)$  diffeomorphism on a  three-dimensional manifold $M$ 
which has a homoclinic tangency of a saddle periodic point $P^{\prime}$ 
with real eigenvalues $\lambda_{1}, \lambda_{2}, \lambda_{3}$ 
satisfying $|\lambda_{1}|<1<|\lambda_{2}|<|\lambda_{3}|$. 
In addition, suppose that the homoclinic tangency satisfies the Tatjer condition. 
Then 
there are  
a two-parameter family $\{f_{a, b}\}_{a,b\in \mathbb{R}}$ with $f_{0,0}=f$ 
and 
a sequence $\{(a_{n}, b_{n})\}_{n\in\mathbb{N}}$ of the parameter values with $(a_{n}, b_{n})\to (0,0)$ as $n\to \infty$
such that, 
for any sufficiently large $n$, 
$f_{a_{n}, b_{n}}$  has an $n$-periodic smooth attracting invariant circle.
\hfill$\square$
\end{lemma}
\noindent
We remark that the attracting invariant circles in Lemma \ref{lem3.2} are really generated by the Hopf  (also known as the Neimark-Sacker) bifurcation for three-dimensional diffeomorphisms,  
which is a part of the codimension-two bifurcation
called Bogdanov-Takens, see 
Broer et al \cite{BRS96}, and also \cite[\S 2.4]{Tj01}. 
\begin{add}[about strange attractors/infinitely many sinks]
\normalfont 
By using other results in \cite[Theorem 1]{Tj01}, one can show that
any element of $\mathscr{C}$ can be approximated by diffeomorphisms having strange attractors or 
infinitely many sinks. However, by \cite[Theorem C]{KS12} together with \cite[Corollary B]{KNS10}, 
both the phenomena can be directly derived from the existence of intrinsic tangencies in the proof of Proposition \ref{prop2.1}
without detouring to the constructing of generalized homoclinic tangencies.
That is, any element of $\mathscr{A}$ can be approximated by diffeomorphisms having these phenomena. 
\hfill$\square$
\end{add}

\section{Constructions of wandering domains}\label{sec4}
Let $\mathscr{D}$ be the set of all  $C^{r}$ 
diffeomorphisms on a three-dimensional manifold $M$
 having  smooth attracting  invariant circles 
which are given by the Hopf  bifurcation
as in the conclusion of Lemma \ref{lem3.2}.  
Here the regularity $r$ should be at least $5$ so that 
it can be expressed by the following normal form (\ref{nf}) of the 
Hopf bifurcation, see \cite[\S 7.5]{Rb}. 
More specifically, for any $f\in \mathscr{D}$, 
there exists a one-parameter family $\{f_{\mu}\}_{\mu\in\mathbb{R}}$ of $C^{r}$ diffeomorphisms
such that
 $f_{0}$ is arbitrarily $C^{5}$-close to $f$ and 
 $f_{\mu}$ undergoes the generic 
Hopf bifurcation at an $n$-periodic point $\boldsymbol{p}$ for some integer $n\geq 0$ for $\mu=0$ 
 which creates an attracting  invariant circle.
In fact,  one has 
polar coordinates $(r, \theta)$ and  
a real coordinate $t$ in a small neighborhood of $\boldsymbol{p}$
such that  $\boldsymbol{p}=(0,0,0)$ and 
\begin{equation}\label{nf}
f_{\mu}^{n}(r, \theta, t)=
\left(
(1+\mu)r-a_{\mu} r^{3}+O_{\mu}(r^{4}),\ 
\theta+\beta_{\mu}+O_{\mu}(r^{2}),\ 
\gamma t 
\right)
\end{equation}
where $O_{\mu}(r^{k})$ is a smooth function of order $r^{k}$ with $k\in\mathbb{N}$ near $(r,\mu)=(0,0)$ 
which depends on $\mu$ smoothly, 
$a_{\mu}$, $\beta_{\mu}$ are real constants depending on $\mu$ smoothly with $a_{0}> 0$, and 
$\gamma$ is a real constant with $0<|\gamma|<1$.
Notice that $f$ restricted to the $r\theta$-space has the same form as 
the normal form of the original two-dimensional Hopf bifurcation, see \cite[\S 7.5]{Rb} 
or \cite[\S\S 4.6--4.7]{Kuz} for more details.
Observe that,  
as $\mu>0$, 
this form has a saddle periodic point at $\boldsymbol{p}=(0,0,0)$ 
and has an attracting invariant circle surrounding $\boldsymbol{p}$
of radius $(\mu a_{\mu}^{-1})^{1/2} +O(\mu)$.

Using the notations $\mathscr{C}$ and $\mathscr{D}$, 
one can rewrite Lemma \ref{lem3.2} as follows: 
\begin{corollary}\label{coro3.3}
$\mathscr{C}$ is contained in the $C^{1}$ closure of $\mathscr{D}$.
\end{corollary}

The next result
is the final step for the proof of Theorem \ref{thm1}.

\begin{proposition}\label{prop4.1}
Let $f$ be a diffeomorphism contained in $\mathscr{D}$. 
There is a diffeomorphism $g$ arbitrarily $C^{1}$-close to $f$ such that 
 $g$ has a contracting non-trivial wandering domain $D$ 
 for which 
 $\omega(D, g)$ is a  transitive nonhyperbolic Cantor set without periodic points. 
\end{proposition}
\begin{proof}
For $f\in \mathscr{D}$, 
one has a $C^r$ diffeomorphism $f_{\mu}$ arbitrarily $C^{1}$-close to $f$ 
such that  $f_{\mu}^{n}$ given by (\ref{nf}) has an attracting invariant circle $S$ 
with $f_{\mu}^{i}(S)\cap S=\emptyset$ for any $i=0,\cdots, n-1$.
By perturbing $f_\mu$ slightly if necessary,
we may assume that $\beta_{\mu}/(2\pi)\not\in\mathbb{Q}$. 
Moreover, since $\mu$ is close to $0$, we 
have a diffeomorphism $\tilde{f}$ $C^{1}$-near $f_{\mu}$ such that 
\begin{equation}\label{rt}
\tilde{f}^{n}(r, \theta, t)=
\left(
(1+\mu)r-a_{\mu} r^{3},\ 
\theta+\beta_{\mu},\ 
\gamma t 
\right)
\end{equation}
in the neighborhood of  $\boldsymbol{p}$ of radius $2(\mu a_{\mu}^{-1})^{1/2}$. 
By this it follows that there is an attracting invariant circle $\tilde S$ and 
the restriction of 
$\tilde{f}^{n}$ to $\tilde S$ is an irrational rigid rotation.
Hence, by the same construction as that of Denjoy's counter-example, 
see \cite{D32} and \cite{Her79}, 
we have  
a $C^{1}$ diffeomorphism $g$ arbitrarily $C^{1}$-close to $\tilde{f}$, 
a sequence $\{ \ell_{i} \}_{i\geq 0}$ of open arcs 
which are contained in a circle $S_{g}$ sufficiently $C^{1}$-close to $\tilde S$ and  
satisfy the following conditions: 
\begin{itemize}
\item 
$S_{g}$ is an attracting invariant circle for $g^{n}$, and 
$g^{n}\vert_{S_{g}}$ is a rigid 
irrational rotation; 
\item for any $i, j\geq 0$ with $i\neq j$, 
\begin{equation}\label{wi}
g^{n}(\ell_{i})=\ell_{i+1},\quad  \ell_{i}\cap \ell_{j}=\emptyset;
\end{equation}
\item $\omega(\ell_{0}, g^{n})$ is a topologically transitive Cantor set on $S_{g}$ without periodic points, 
where $g^{n}$ has zero Lyapunov exponent.
\end{itemize}
%%%%%%%%%%%%%%%%%%%%%%%%%%%%%%%%%%%%%%%
\begin{figure}[hbt] 
\centering
\scalebox{0.8}{\includegraphics[clip]{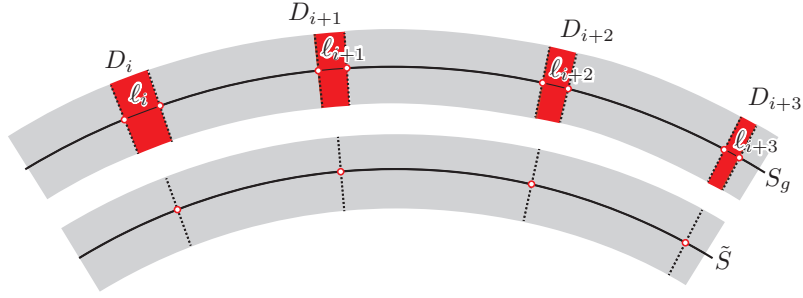}}
\caption{Denjoy-like construction}
\label{fg6}
\end{figure}
%%%%%%%%%%%%%%%%%%%%%%%%%%%%%%%%%%%%%%%

We here consider a normal tubular neighborhood of each arc $\ell_{i}$, see Figure \ref{fg6}, which is  defined as
$$D_{i}:=\bigcup_{x\in \ell_{i}} \Delta_{i}(x),$$ 
where $\Delta _i(x)$ is the open disk of radius $\delta$ centered at $x\in \ell _i$ which lies in a plane normal to $\ell _i$ for each $i\geq 0$, where $\delta >0$ is a given  small number independent of $x$ and $i$.
By the form of (\ref{nf}), the restrictions of $f_\mu^n(r,\theta,t)$ to the first and third entries 
are contracting maps.
It follows from this fact together with the wandering condition (\ref{wi}) 
that 
$$ g^{n}(D_{i})\subset D_{i+1},\quad 
D_i\cap D_j=\emptyset
$$ 
for every $i, j \geq 0$ with $i\neq j$. 
In consequence, the open set $D_{0}$ is a 
contracting  non-trivial wandering domain for $g^{n}$. 
Since $g^{i}(S_{g})\cap S_{g}=\emptyset$ for $i=1,\cdots, n-1$, 
$D_{0}$ is a contracting non-trivial wandering domain also for $g$.
\end{proof}
We are at last ready to show the main result of this paper.
\begin{proof}[Proof of Theorem \ref{thm1}]
By Propositions \ref{prop2.1}, \ref{thm3.1}, and Corollary \ref{coro3.3}, 
$$
\mathscr{A}\subset 
\overline{(\mathscr{B})}_{C^{1}}, \quad 
\mathscr{B}\subset \overline{(\mathscr{C})}_{C^{r}}\subset 
\overline{(\mathscr{D})}_{C^{r}}. 
$$
%$$
%\mathscr{A}\subset 
%\overline{(\mathscr{B})}_{C^{1}} \subset 
%\overline{(\mathscr{C})}_{C^{1}}\subset 
%\overline{(\mathscr{D})}_{C^{1}}. 
%$$
Hence, for any $f\in \mathscr{A}$, 
there is an  $\hat{f}\in \mathscr{D}$  arbitrarily $C^{1}$-close to $f$. 
By Proposition \ref{prop4.1}, 
we obtain a diffeomorphism $g$ arbitrarily $C^{1}$-close to $\hat{f}$ 
which has a contracting non-trivial wandering domain. 
It implies that $g$ is an element of $\mathscr{Z}$.
This completes the proof.
\end{proof}

\section*{Acknowledgements.}
This paper was partially supported by 
JSPS KAKENHI Grant Numbers 25400112 and 26400093.
The authors thank hospitality of the Kyoto Dynamical Seminar.
The authors also thank helpful comments and  suggestions 
by Alexandre Rodrigues and the anonymous referees.

%%%%%%%%%%%%%%%%%

\end{document}